\documentclass[12pt,a4paper,oneside]{amsart}

\usepackage{amsfonts, amsmath, amssymb, amsthm, amscd}
\usepackage{hyperref}
\hypersetup{
  colorlinks   = true, 
  urlcolor     = blue, 
  linkcolor    = blue, 
  citecolor   = red 
}
\usepackage{anysize}
\marginsize{2cm}{2cm}{1.5cm}{1.5cm}
\usepackage[pdftex]{graphicx}

\newtheorem{theorem}{Theorem}
\newtheorem{lemma}[theorem]{Lemma}

\newtheorem{corollary}[theorem]{Corollary}

\theoremstyle{definition}

\theoremstyle{remark}
\newtheorem{remark}[theorem]{Remark}

\numberwithin{equation}{section}

\DeclareMathOperator{\vol}{vol}

\newcommand{\lmink}{\mathop{\underline{\mathcal M}}\nolimits}

\renewcommand{\epsilon}{\varepsilon}

\renewcommand{\phi}{\varphi}

\author{Arseniy~Akopyan{$^\spadesuit$}}

\email{akopjan@gmail.com}

\author{Roman~Karasev{$^\clubsuit$}}

\email{r\_n\_karasev@mail.ru}
\urladdr{http://www.rkarasev.ru/en/}

\title{A tight estimate for the waist of the ball}

\thanks{$^\spadesuit$ Supported by People Programme (Marie Curie Actions) of the European Union's Seventh Framework Programme (FP7/2007-2013) under REA grant agreement n$^\circ$[291734].}

\thanks{{$^\clubsuit$} Supported by the Russian Foundation for Basic Research grant 15-31-20403 (mol\_a\_ved).}
\thanks{{$^\clubsuit$} Supported by the Russian Foundation for Basic Research grant 15-01-99563 A}

\address{{$^\spadesuit$} Institute of Science and Technology Austria (IST Austria), Am Campus 1, 3400 Klosterneuburg, Austria}
\address{{$^\clubsuit$} Moscow Institute of Physics and Technology, Institutskiy per. 9, Dolgoprudny, Russia 141700}
\address{{$^\clubsuit$} Institute for Information Transmission Problems RAS, Bolshoy Karetny per. 19, Moscow, Russia 127994}

\begin{document}

\begin{abstract}
	We answer a question of M. Gromov on the waist of the unit ball.
\end{abstract}
\maketitle

In \cite{grom2003} Gromov proved that for any continuous map $f : \mathbb R^n \to \mathbb R^k$, where we fix a Gauss measure $\gamma$ in $\mathbb R^n$, there exists a fiber $f^{-1}(y)$ such that any its $t$-neighborhood has measure $\gamma(f^{-1}(y)+t)$ (we denote $t$-neighborhood simply by $+t$) at least the measure of a $t$-neighborhood of a linear subspace $\mathbb R^{n-k}\subset\mathbb R^n$. He also sketched the proof of a similar fact for the round sphere $\mathbb S^n$ with its uniform measure, later presented in a detailed form by Memarian~\cite{mem2009}. Those results were strong versions of so called \emph{waist} inequalities, that is inequalities that guarantee that some fiber $f^{-1}(y)$ is big in a certain sense.

One reasonable way to measure fibers is to use its lower $(n-k)$-dimensional Minkowski content defined by
\[
\lmink_{n-k} f^{-1}(y) := \liminf_{t\to+0} \frac{\vol_n (f^{-1}(y)+t)}{v_k t^k},
\]
where $n$ is the dimension of the ambient Riemannian manifold, $\vol_n$ is a the Riemannian volume of this ambient space, $v_k$ is the volume of the unit Euclidean $k$-dimensional ball. The results of Gromov and Memarian directly imply that for any continuous $f : \mathbb S^n\to\mathbb R^k$ some fiber $f^{-1}(y)$ has the lower Minkowski content at least that of the standard equatorial $\mathbb S^{n-k}\subset \mathbb S^n$.

Not much was known about the (Minkowski content) waists of spaces other than spheres, already the cases of the unit cube and the unit Euclidean ball were not known until recently. Answering a question by Gromov and Guth and generalizing the theorem of Vaaler~\cite{vaaler1979}, Klartag showed in~\cite{klartag2016} that for any continuous map $f : (0,1)^n \to \mathbb R^k$ some fiber $f^{-1}(y)$ has
\[
\lmink_{n-k} f^{-1}(y) \ge 1.
\]
The idea was to: (A) transport a Gaussian measure with density $e^{-\pi |x|^2}$ on the whole $\mathbb R^n$ to $(0,1)^n$ by a map $T : \mathbb R^n\to (0,1)^n$ defined coordinate-wise; (B) use the mentioned theorem of Gromov to estimate the measure of the $t$-neighborhood of the fiber of the composed map $f\circ T$; (C) use the Lipschitz property of $T$ to have the similar estimate for a $t$-neighborhood of $f^{-1}(y)$; (D) go to the limit $t\to 0$, as it is required in the definition of the Minkowski measure.

We use an analogous argument to answer another question of Gromov \cite[Question 3.1.A]{grom2003} about the ball in place of the cube:

\begin{theorem}
\label{thm:ball-waist}
Let $B^n\subset \mathbb R^n$ be the round unit ball. Then for any continuous $f : B^n\to \mathbb R^k$ some fiber $f^{-1}(y)$ has $(n-k)$-dimensional lower Minkowski content at least 
\[
v_{n-k} = \vol_{n-k} B^{n-k}, 
\]
and the estimate is evidently tight for linear $f$.
\end{theorem}

\begin{remark}
	\label{rem:waist to manifold}
	We could invoke the generalization~\cite{karvol2011} of the sphere waist theorem for continuous maps to arbitrary $k$-manifolds in place of $\mathbb R^k$ resulting in replacing $\mathbb R^k$ with an arbitrary $k$-manifold in this theorem. 
\end{remark}

\begin{proof}[Proof of the theorem]
Consider the orthogonal projection $P : \mathbb S^{n+1}\to B^n$, viewing $\mathbb S^{n+1}\subset\mathbb R^{n+2}$ as the standard unit sphere. It is known (the theorem of Archimedes~\cite{archimedes225onsphere}, see also \cite[Lemma 3]{bezdek2002pushing}) that this map sends the uniform measure $\sigma$ on $\mathbb S^{n+1}$ to the uniform Lebesgue measure in $B^n$ up to constant $2\pi$. The constant can be seen as follows: Let $v_n$ be the total volume of $B^n$ and let $s_{n+1} = (n+2)v_{n+2}$ be the total measure of $\mathbb S^{n+1}$; from the formula $v_\ell = \frac{\pi^{\ell/2}}{\Gamma(\ell/2+1)}$ we infer $s_{n+1} = 2\pi v_n$.

Now apply the waist of the sphere theorem of Gromov and Memarian~\cite{grom2003,mem2009} to find $y\in \mathbb R^k$ such that any $t$-neighborhood of $(f\circ P)^{-1}(y)$ has measure at least the measure of the $t$-neighborhood of the equatorial $(n+1-k)$-sphere $\mathbb S^{n+1-k}\subset \mathbb S^{n+1}$. We are not going to write down the formula, but in the asymptotic form this is going to be (we assume $t\to +0$)
\[
\sigma \left( (f\circ P)^{-1}(y) + t \right) \ge s_{n+1-k} v_k t^k (1 + o(1)).
\]
Since the projection $P$ is evidently $1$-Lipschitz, the $t$-neighborhood of $f^{-1}(y)$ contains the projection of the $t$-neighborhood of $(f\circ P)^{-1}(y)$, so we have (not forgetting the normalization)
\[
\vol_n \left( f^{-1}(y) + t \right) \ge 2\pi s_{n+1-k} v_k t^k (1 + o(1)).
\]
Passing to the limit, we obtain for the lower Minkowski content
\[
\lmink_{n-k} f^{-1}(y) := \liminf_{t\to+0} \frac{\vol_n \left( f^{-1}(y) + t \right)}{v_k t^k} \ge 2\pi s_{n+1-k} = \frac{1}{2\pi} \cdot 2\pi \cdot v_{n-k} =v_{n-k}.
\]

%
\end{proof}

\begin{remark}
	\label{rem:gromov}
	In \cite[Question 3.1.A]{grom2003} Gromov generally asked about the assumptions on a rotation invariant measure in $\mathbb{R}^n$ that guarantee a waist inequality with extremum attained at the $(n-k)$-plane through the origin, noting that this was not even known for the uniform measure in the ball. Theorem~\ref{thm:ball-waist} answers the latter question about the ball affirmatively, if we understand the $(n-k)$-volume of a fiber as the Minkowski measure. Extending the argument in the straightforward way, it is possible to show that the waist is attained at the $(n-k)$-plane through the origin for all measures in $\mathbb{R}^n$ obtained by orthogonal projection of the uniform measure of a sphere in higher dimension $\mathbb{R}^{N}$, $N>n$. 
	
	There remains an open question about the waist theorem for the ball in the sense of Gromov--Memarian, estimating the $n$-volume of any $t$-neighborhood of a fiber. Our proof gives some inequality by taking the estimate from the sphere, but this inequality is evidently not tight, unless we go to the limit $t\to +0$.
\end{remark}

Now we follow~\cite[Corollary 5.3]{klartag2016} and prove

\begin{corollary}
\label{cor:ellipsoid}
Consider an ellipsoid $E$ with principal axes $a_1\le a_2\le \dots \le a_n$ and a continuous map $f : E\to \mathbb R^k$. There exists a fiber $f^{-1}(y)$ with $\lmink_{n-k} f^{-1}(y) \ge v_{n-k}\prod_{i=1}^{n-k} a_i$.
\end{corollary}

The proof of this corollary consists in considering the evident linear map $L : B^n \to E$. applying Theorem~\ref{thm:ball-waist} to $f\circ L$ to have $\lmink_{n-k} (f\circ L)^{-1}(y) \ge v_{n-k}$ and applying the following theorem to the map $L$:

\begin{theorem}
\label{thm:minklinear}
The linear map $L(x_1,\ldots,x_n) = (a_1x_1,a_2x_2,\ldots,a_nx_n)$ with $0<a_1\le a_2\le \dots \le a_n$ has the property
\[
a_1\ldots a_k \lmink_k X \le  \lmink_k L(X) \le a_{n-k+1}\dots a_n \lmink_k X.
\]
\end{theorem}

We will prove the upper bound, the lower bound follows from the upper bound for $L^{-1}$. The proof of this theorem decomposes into several lemmas.

\begin{lemma}
\label{lem:scale}
If $L$ is a homothety with factor $\lambda$ then $\lmink_k L(X) = \lambda^k \lmink_k X.$
\end{lemma}

\begin{proof}
Follows from the facts that $L(X+t) = L(X) + \lambda t$ and $\vol_n L(X+t) = \lambda^n \vol_n (X+t)$.
\end{proof}

\begin{lemma}
\label{lem:stretch}
If $L(x_1,\ldots,x_n) = (ax_1, x_2,\ldots, x_n)$ and $a\ge 1$ then $\lmink_k L(X) \le a \lmink_k X$.
\end{lemma}

\begin{proof}
We observe that $\vol_n L( X + t) = a \vol_n (X+t)$ and $L( X + t ) \supseteq L(X) + t$, and deduce 
\[
\vol_n (L(X) + t) \le a \vol_n (X + t).
\]
\end{proof}

\begin{lemma}
\label{lem:shrink}
If $L(x_1,\ldots,x_n) = (ax_1, x_2,\ldots, x_n)$ and $0<a\le 1$ then $\lmink_k L(X) \le \lmink_k X$.
\end{lemma}

\begin{proof}
We show that the intersection of $X + t$ with any line $\ell$ parallel to $0x_1$ axis has length at least that of $(L(X)+t)\cap \ell$ and the result follows by the Fubini theorem. Every point $x\in X$ contributes a segment into $(X+t)\cap \ell$, the point $L(x)$ contributes to $(L(X)+t)\cap \ell$ a segment of the same length. So we have one set of segments of given lengths and the other set of segments of the same length with centers shrunk with factor $a$. So we have one set of segments of given lengths and the other set of segments of the same length with centers shrunk with factor $a$; and need to show that the total length of the union of segments is greater for the first family. Indeed,\footnote{This was prompted to us by Ilya Bogdanov.} any connected component of the first union of segments does not increase in diameter when passing to the second family of segments, and the components may merge, thus the total length cannot increase.
\end{proof}

Now to prove Theorem~\ref{thm:minklinear} we represent $L$ as the composition of

1) the homothety with coefficient $a_{n-k+1}$, Lemma~\ref{lem:scale} gives factor $a_{n-k+1}^k$;

2) the scalings of the last $k-1$ coordinates by $a_{n-k+2}/a_{n-k+1}, \ldots, a_n/a_{n-k+1}$, Lemma~\ref{lem:stretch} accumulates the factor $a_{n-k+2}\dots a_n / a_{n-k+1}^{k-1}$;

3) the scalings of the first $n-k$ coordinates by $a_1/a_{n-k+1}, \ldots, a_{n-k}/a_{n-k+1}$, Lemma~\ref{lem:shrink} shows $\lmink_k$ does not increase.

The total collected factor is then $a_{n-k+1}\dots a_n$, proving the upper bound in Theorem~\ref{thm:minklinear}.

Evidently, Theorem~\ref{thm:minklinear} produces the following improvement of~\cite[Corollary 5.3]{klartag2016}:

\begin{corollary}
\label{cor:parallelotope}
Consider a parallelotope $P = (0,a_1)\times \dots \times (0, a_n)$ with dimensions $0<a_1\le a_2\le \dots \le a_n$ and a continuous map $f : P\to \mathbb R^k$. There exists a fiber $f^{-1}(y)$ with $\lmink_{n-k} f^{-1}(y) \ge \prod_{i=1}^{n-k} a_i$.
\end{corollary}

\bibliography{../Bib/karasev}
\bibliographystyle{abbrv}
\end{document}